\documentclass{amsart}

\usepackage{ amsthm, commath}
\usepackage{amssymb, amsmath, amsfonts, stmaryrd, mathrsfs, amscd}
\usepackage[sans]{dsfont}
\usepackage{fouridx, chemarrow}
\usepackage[dvips, cmtip, all]{xy}
\usepackage{setspace, paralist}
\usepackage[linkcolor=black]{hyperref}
\usepackage{lineno, color}
\usepackage{tikz, mathabx, tikz-cd} 
\usetikzlibrary{matrix, arrows, decorations.pathmorphing}
\usepackage{latexsym}
\usepackage{graphicx,subcaption}
\usepackage{epsfig}

\newtheorem{thm}{Theorem}[section]
\newtheorem{cor}[thm]{Corollary}
\newtheorem{lem}[thm]{Lemma}
\newtheorem{lem-def}[thm]{Lemma-Definition}
\newtheorem*{lem*}{Lemma}
\newtheorem{prop}[thm]{Proposition}
\newtheorem*{thm*}{Theorem}
\newtheorem*{prop*}{Proposition}
\theoremstyle{definition}
\newtheorem{defn}[thm]{Definition}

\theoremstyle{remark}
\newtheorem{rem}[thm]{Remark}
\newtheorem{obs}[thm]{Observation}

\newtheorem{que}[thm]{Open Question}

\def\ZZ{\mathbb{Z}}
\def\QQ{\mathbb{Q}}
\def\RR{\mathbb{R}}

\def\Alphabet{A} 
\def\AA{\mathcal{A}}
\def\AA{\gring{k}{F}}
\newcommand{\ve}[1]{\mathbf{#1}} 
\def\uu{\ve{u}}
\def\vv{\ve{v}}
\def\ww{\ve{w}}
\renewcommand{\tt}{\ve{t}}

\def\Ocal{\mathcal{O}}

\newcommand{\tsr}{\otimes}
\newcommand{\zero}{^{(0)}}
\newcommand{\one}{^{(1)}}
\newcommand{\two}{^{(2)}}
\newcommand{\onezero}{^{(10)}}
\newcommand{\oneone}{^{(11)}}
\newcommand{\onetwo}{^{(12)}}

\newcommand{\isoto}{\xrightarrow{\cong}}
\renewcommand{\bar}{\overline}
\newcommand{\ab}{^{\mathrm{ab}}}

\newcommand{\wozero}{\backslash\{0\}}
\newcommand{\gring}[2]{#1#2}
\newcommand{\without}{\backslash}
\newcommand{\rprod}{\cdot}
\renewcommand{\phi}{\varphi}
\newcommand{\inject}{\hookrightarrow}
\newcommand{\thecurve}{W}
\newcommand{\vecspan}[2]{#1\mbox{-span}#2}
\DeclareMathOperator{\supp}{Supp}
\DeclareMathOperator{\Frac}{Frac}
\DeclareMathOperator{\Spec}{Spec}
\DeclareMathOperator{\rank}{rank}

\makeatletter
\newsavebox{\@brx}
\newcommand{\llangle}[1][]{\savebox{\@brx}{\(\m@th{#1\langle}\)}%
\mathopen{\copy\@brx\kern-0.5\wd\@brx\usebox{\@brx}}}
\newcommand{\rrangle}[1][]{\savebox{\@brx}{\(\m@th{#1\rangle}\)}%
  \mathclose{\copy\@brx\kern-0.5\wd\@brx\usebox{\@brx}}}
\makeatother

\title{Centralizers in Free Group Algebras and Nonsingular Curves}
\author{Nikita Miasnikov}
\address{SUNY Oswego\\
Oswego, NY 13126}
\email{nikita.miasnikov@oswego.edu}

\begin{document}
\bibliographystyle{alpha}
\date\today
\begin{abstract}
The centralizer of any non-scalar element of a free group algebra over a field is the coordinate ring of a nonsingular curve.
\end{abstract}

\maketitle
%
%
\section{Introduction}\label{intro}
Let $k$ be a field with algebraic closure $\bar k$.  G. Bergman showed in~\cite{Bergman} (see also ~\cite{Lothaire} for an exposition) that the centralizer $C$ of a non-scalar element $\uu$ in a free associative algebra over $k$ is isomorphic to $k[t]$, the ring of polynomials in one indeterminate over $k$.  His proof can be split into two parts.

First he showed~\cite[Lemma 1.5, Prop. 2.2 and Prop. 4.5]{Bergman} that, to paraphrase the original formulation, $C$ is {\em the affine coordinate ring of the complement of a $k$-point in a proper nonsingular curve over $k$.}
Such rings are discussed in Section~\ref{sec:niceCurves}.  

The second part of his proof went as follows.  Bergman proved~\cite[Prop. 5.2]{Bergman} that  {\em Every finitely generated subalgebra $R\neq k$ of a free $k$-algebra admits a homomorphism $f:R\to k[t]$ which is non-trivial in the sense $f(R) \neq k$. } Since from the first part of the proof one knows that $C$ is a $1$-dimensional integral domain, it follows that any non-trivial homomorphism $C\to k[t]$ must be an embedding.  To complete the argument, Bergman invoked the following result due to P.M. Cohn~\cite[Prop. 2.1]{CohnSubalgebras} which uses Luroth's Theorem: {\em Any integrally closed subalgebra of $k[t]$, other than $k$, is of the form $k[y]$.}


Fix a finite alphabet $\Alphabet$.
The free associative algebra $k\langle A\rangle$ on $\Alphabet$ is contained in the free skew field $k(\langle A\rangle)$ on the same alphabet over $k$.   
We can think of the elements of $k\langle A\rangle$ as non-commutative polynomials and of elements of $k(\langle A\rangle)$ as non-commutative rational functions.   
%
Perhaps the next simplest subalgebra of the free skew field is the $k$-algebra generated by the elements of $\Alphabet$ and their inverses.   This is none other than the free group algebra on the set of generators $\Alphabet$, and its elements can be thought of as non-commutative Laurent polynomials. 

Let $F$ be a free group on $\Alphabet$.  
In this paper, we transplant the first part of Bergman's proof into the setting of the free group algebra $\gring{k}{F}$.  Our main results are Proposition~\ref{prop:hg}  and Theorem~\ref{theorem}.  

Theorem~\ref{theorem} combined with Corollary~\ref{Laurent} give the following description of the centralizer $C$ of a non-scalar element $\uu\in \gring{k}{F}$.  

{\em Either $C$ is isomorphic to the ring of Laurent polynomials $k[t,t^{-1}]$ (Corollary ~\ref{Laurent}), or else $C$ is the affine coordinate ring of the complement of a $k$-point in a proper nonsingular curve $\thecurve$ over $k$.
 (Theorem~\ref{theorem}). }
 
 Proposition~\ref{prop:hg} says that for a class of elements, which we call {\em homogeneous}, the centralizer $C\cong k[t]$, i.e. the curve $\thecurve$ is in fact a line. 
 
We do not have an example in which that curve is not a line.

\section{Acknowledgments}
I thank the referee, Marian Anton, Maggie Habeeb, Andrew Stout, and Phillip Williams for all of their comments and suggestions.   I thank my father for too many things to list.
%
%
\section{Preliminaries}\label{prelim}

\subsection{Free group algebras}\label{fga}
Let $F$ be a free group on a finite alphabet $\Alphabet$.  For $u \in F$, we denote by $\bar{u}$ the {\em reduced word} in the alphabet $\Alphabet^{\pm 1}:=\{a, a^{-1} | a\in \Alphabet\}$  which {\em represents $u$}. 


 We denote by $\gring{k}{F}$ the group algebra of $F$ over $k$.  By definition, $\AA$ is a $k$-vector space.  One thinks of $F$ as a subset of $\AA$, and in fact, $F$ is a $k$-basis of $\AA$.  Moreover, $\AA$ is equipped with multiplication which linearly extends multiplication of $F$.   Every element $\uu$ of $\AA$ can be uniquely written as 
\[\uu=\sum_{g\in F} a_g g\]
where $a_g\in k$ and only finitely many of $a_g$ are nonzero.
For $\uu\in \AA$,  $ \supp(\uu) $ will denote the {\em support } of $\uu$, i.e. the set of elements of $F$  which appear with non-zero coefficients in $\uu$.   
By a {\em monomial} or {\em term} we mean an element of $\AA$ whose support consists of one element.

 
Let $X= \{x_a | a\in A\}$ be distinct letters indexed by the alphabet $\Alphabet$.  A very useful fact about $\AA$ due to W. Magnus and R. Fox~\cite[Theorem 4.3]{Fox}, which we use, is that $\AA$ can be embedded into the algebra of non-commutative formal power series $k\llangle X\rrangle$ in the indeterminates $x_a$ ($a\in\Alphabet$). This embedding, called the {\em Magnus-Fox embedding}, is effected by sending each letter $a\in A$ to $1+x_a$.   Note that this embedding induces an isomorphism onto  $k\llangle X\rrangle$ of the completion of $\AA$ with respect to the augmentation ideal. 

A free group can be bi-ordered by pulling back an ordering along the embedding $\gring{\ZZ}{F}\to \ZZ\llangle X\rrangle$.   For any bi-orderable group $G$, it is easy to show, that  $\gring{k}{G}$ {\em is an integral domain in which the only units are the monomials}.  In particular, $\AA$ has these properties.  

Moreover, $\AA$ is a {\em free ideal ring} or {\em  fir}~\cite[Corollary 7.11.8, page 507]{Cohn}.  A fir is a ring in which every one-sided ideal, left or right, is a free module of a unique rank (see ~\cite[Ch. 2]{Cohn} for an exposition of firs).

\subsection{Degree functions}\label{degFun}
\begin{defn} \label{def:degFun}Given a $k$-algebra $R$, a {\em degree function} on $R$ over $k$  is a map 
\[d:R\to \{-\infty\}\cup \RR\]
 which satisfies the properties: $\forall x,y \in R$,
\[d(x)=-\infty\iff x=0, \quad\quad x\in k\wozero \implies d(x)=0\]
 \begin{align}
 d(xy)&= d(x) + d(y)\label{degmult}\\
d(x+y)  &\leq \max\{d(x), d(y)\}\label{eq:ultratriangle}
\end{align}
A degree function is {\em discrete} if it maps $R\wozero$ into a cyclic subgroup of $(\RR,+)$.
\end{defn}

\begin{rem} \label{rem:valuations} Identity~(\ref{degmult}) shows that any $R$ that admits a degree function is an integral domain (i.e. $R$ has no zero-divisors).   When $R$ is a {\em commutative} integral domain, $d$ is a degree function if and only if $-d$ extends to a real-valued valuation on $\Frac C$, the field of fractions of $C$.   In fact $d$ {\em determines} the resulting valuation $v$ because $-d$ extends to $v$ {\em uniquely} via
\[v\left(\frac{x}{y}\right)= d(y)-d(x).\]   
Finally $v$ is discrete whenever $d$ is discrete and non-trivial (i.e. not identically zero on $R\wozero$).
\end{rem}

\begin{rem}\label{rem:isosceles} The following property, usually seen with ultra-metrics (and interpreted as ``triangles are isosceles"), holds here as well:
\[d(y)< d(x) \implies d(x+y) = d(x).\]
Indeed, denote $z=x+y$.  By (\ref{eq:ultratriangle}), $d(z)\leq d(x)$.  If $d(z)<d(x)$, then $x=z-y$ gives $d(x)\leq \max\{d(z), d(-y)\}< d(x)$, a contradiction. 
\end{rem}

A group homomorphism  
\[h: F\to (\RR,+)\]
defines a grading on $\AA$ with indices in $\RR$.  Thus 
\[\AA = \bigoplus_{a\in \RR} \AA_a\quad\mbox{ where }\AA_a = \vecspan{k}{(h^{-1}(a))}.\]
   Note that multiplication in $\AA$ is compatible with grading: $(\AA_a) (\AA_b) \subseteq \AA_{a+b}$.  The $k$-vector spaces $\AA_a$ are the {\em $h$-homogeneous} components of $\AA$. If $\uu \in (\AA)_a$ is non-zero, we say that $\uu$ is {\em $h$-homogeneous}.   For $\uu\in \AA\wozero$, we denote by $\uu_h$ its {\em highest-degree homogeneous component}.   A key observation is that since $\AA$ is an integral domain,
\begin{equation}
(\uu\vv)_h=\uu_h\vv_h.\label{eq:key}
\end{equation}
We extend $h$ to $\AA$ by setting
\[h(\uu)= h(\uu_h)=\max\{ h(u) | u\in \supp(\uu)\}, \quad h(0)=-\infty.\]
With the aid of (\ref{eq:key}) one can see that the resulting map $h:\AA\to\{-\infty\}\cup\RR$ is a degree function on $\AA$ over $k$.  We refer to this degree as {\em $h$-degree}. When no confusion is possible we may drop explicit references to the homomorphism $h$ and use the plain terms {\em homogeneous} and {\em degree}.

\subsection{Basic properties of centralizers}\label{centralizerBasics} In $\AA$, {\em the centralizer $C$ of an element $\uu\not\in k$   is a commutative and integrally closed integral domain}. 
Commutativity follows from the Magnus-Fox embedding of $\AA$ into the non-commutative power series ring $k\llangle X\rrangle$ (cf. Section~\ref{fga}), since in $k\llangle X\rrangle$, the centralizer of any element which is not contained in $k$ is isomorphic to the ring of formal power series $k\llbracket t\rrbracket$ (See~\cite[Corollary 6.7.2, p. 375]{Cohn}).

The assertion that $\AA$ is integrally closed is a consequence of the following proposition due to G. Bergman~\cite[Corollary 4.4]{Bergman} which relies on the notion of a $2$-fir.  For any cardinal $n$, a {\em right $n$-fir} is a ring in which all right ideals generated by $n$ or fewer elements are free and have unique rank. For finite $n$,  the notions of left $n$-fir and right $n$-fir are equivalent (\cite[Ch 2]{Cohn}).  Clearly any fir is a $2$-fir.
\begin{prop}[G. Bergman]  Let $R$ be a $k$-algebra which is an integral domain, such that $k(t)\tsr R$ is a $2$-fir and such that $R$ remains an integral domain under tensoring with all finite algebraic extensions of $k$.   Let $C$ be a commutative subring of $R$ and $\widetilde{C}$ the integral closure of $C$ in its field of fractions.  Then the inclusion of $C$ in $R$ extends uniquely to an embedding of $\widetilde{C}$ in $R$.
\end{prop}

The centralizer {\em group} $Z$ of an element $g\in F$, $g\neq 1$ is cyclic.  Indeed, $Z$ is a commutative subgroup of $F$, as can be seen, for example, using the Magnus-Fox embedding; and it is also free, since by Nielsen-Schreier theorem any subgroup of a free group is free.    That makes $Z$ cyclic, necessarily the maximal cyclic subgroup of $F$ which contains $g$.

The next lemma allows one to assume that a free group is finitely generated when dealing with centralizers.

\begin{lem} \label{finiteAlphabet} Suppose that $F$ is a free group of arbitrary, possibly infinite, rank.  Suppose that $G$ is a free factor of $F$, i.e.  that $F=G\!*\!H$, the free product of the free groups $G$ and $H$.  Let $\uu\in \gring{k}{G}$, $\uu\not\in k$.   Then the centralizer of $\uu$ in $\gring{k}{F}$ is contained in $\gring{k}{G}$.
\end{lem}

\begin{proof}Suppose there is an element $\vv\in\gring{k}{F}\without kG$ which commutes with $\uu$.   We will conclude that $\uu\in k$, a contradiction.  It is enough to consider the case when $G$ and $H$ are finitely generated because we only need enough generators to express both $\uu$ and $\vv$.  

Let us bi-order the free group $F$ (cf. Section~\ref{fga}).  Denote by $u$ the lowest element in $\supp(\uu)$ and by $v$ the lowest element in $\supp(\vv)\without G$.  Since a bi-ordering is preserved by multiplication from both sides, the lowest element of $\supp(\uu\vv) \without G= \supp(\vv\uu)\without G$ is both $uv$ and $vu$, and so $u$ and $v$ commute.  That means $u=r^m$ and $v=r^n$ for some element $r\in F$.  $v\not\in G$ implies that $r\not\in G$.   We will now deduce that $m=0$.  

Let $\{g_1,\dots\}$ be a free generating set for $G$, $\{t_1,\dots\}$ a free generating set for $H$.  Working in the free generating set $\{g_1,\dots, t_1,\dots\}$ for $F$, write 
\[r= p^{-1}s p\]
 so that $\bar{p\;}^{\;-1} \;\bar{s} \;\bar{p}$ is a reduced word and $\bar{s}$ cyclically reduced, i.e. so that the first letter of $\bar{s}$ is not the inverse of the last letter.  Since $r\not\in G$, at least one of $\bar{p}$ or $\bar{s}$ must contain a letter from $\{t_1,\dots\}$.
 
Now, $r^m$ is represented by the word $\bar{p\;}^{\;-1}\; \bar{s\;}^{\;m} \;\bar{p\;}$.  If $m\neq 0$,  this word is reduced and contains a letter from $\{t_1,\dots\}$, which implies that $r^m\not\in G$. Since $r^m=u\in G$, it follows that $m=0$ and $u=1$.
Thus the lowest element in $\supp(\uu)$ is $1$.  Similarly we conclude that the highest element in $\supp(\uu)$  is $1$.  Hence $\uu\in k$, a contradiction. \end{proof}
\subsection{The complement of a regular $k$-point in a proper curve over $k$.}\label{sec:niceCurves}
Parts of the following characterization are in Propositions 2.1 and 2.2 in~\cite{Bergman}. We prove this lemma in Section~\ref{sec:proofCurves} for the sake of completeness of the exposition.  See~\cite[Sec. 4.1, in particular Exercise 1.17]{Liu} for a treatment of algebraic curves over an arbitrary field.   We denote by $\Frac C$ the field of fractions of $C$.
 
\begin{lem}\label{characterization} Let $C\neq k$ be a $k$-algebra which is a commutative integral domain.  Then the following are equivalent:
\begin{compactenum}[(i)]
\item There is a discrete degree function (cf. Definition~\ref{def:degFun})  $d$ on $C$  over $k$ which is non-negative on $C\wozero$ and which satisfies: 
  \begin{equation}\label{eq:totallyRamified}
     \left(\forall  x,y \in C\wozero\right)\quad    d(x)= d(y)                \implies \left(\exists \lambda \in k\right)\;\; d(x-\lambda y)<d(x).
\end{equation}
 \item For any $x\in C\without k$ the following holds: $x$ is transcendental over $k$,  $C$ is a finitely generated $k[x]$-module and  the valuation of $k[x]$ at $\infty$ is totally ramified in $\Frac C$.
\item There is $x\in C\without k$ for which the following holds: $x$ is transcendental over $k$, $C$ is a finitely generated $k[x]$-module and the valuation of $k[x]$ at $\infty$ is totally ramified in $\Frac C$. 
\item There is a proper reduced and irreducible curve $\thecurve$ over $k$ and a regular point $p\in \thecurve$ defined over $k$ (which implies that $\thecurve$ is geometrically reduced and irreducible and smooth at $p$) such that $C\cong \Ocal_\thecurve(\thecurve\without \{p\})$. 
\item There is a discrete valuation $v$ of $\Frac C$ over $k$ which is non-positive on $C\wozero$ and whose residue field is $k$.
\end{compactenum}
\end{lem}
\begin{rem}Any $k$-algebra $C$ safisfying condition i) of the lemma is automatically a commutative integral domain~\cite[Proposition 2.1]{Bergman}.
\end{rem}
\begin{rem}\label{rem:uniquePlace}Suppose $C$ satisfies the conditions of the lemma and $d'$ is a degree function on $C$ over $k$ with at least one positive value.  Then for some $\lambda>0$, $d'\equiv \lambda d$.  This is due to the fact that $-d$ and $-d'$ both are valuations (cf. Remark~\ref{rem:valuations}) associated to the point ``at infinity" $p$.
\end{rem}

%
%
 \section{A linear algebra lemma}

\begin{lem} \label{linAlgLem} Let $U, V,Z$ be vector spaces over $k$ and $\mu: V\times U\to Z$  a bilinear map without zero divisors over $\bar k$ , i.e., such that if we denote $\mu'= \bar k\tsr\mu$, then whenever $\mu'(x,y )=0$ with $x \in \bar k\tsr U$ and $y\in \bar k\tsr V$ it follows that $x=0$ or $y=0$.
 Consider the following bilinear product induced by $\mu$:
\[*:  (U\tsr V)\times( U\tsr V) \to U \tsr Z\tsr V,\]
\[    (x\tsr y)* (z\tsr u) = x\tsr \mu(y,z)\tsr u.\]
If $t* s = s* t$ for $t,s\in U\tsr V$,  then $t, s$ are linearly dependent over $k$.
\end{lem}
The following proof is due to the referee.  A proof using duality is also possible.
\begin{proof}

 Since linear dependence of $t,s$ over $\bar k$ is equivalent to their dependence over $k$, we can assume without loss of generality that $k=\bar k$.  Assume also without loss of generality again that the vector spaces $U$ and $V$ are finite-dimensional.

Given $s\in U\otimes V$, one can write 
\begin{equation}\label{eq:tensor}
s = \sum_1^n u_i\otimes v_i
\end{equation}
with $u_1,\dots,u_n$ linearly independent, and $v_1,\dots,v_n$
likewise linearly independent; and though that expression is not
unique, the subspace $U_s$ of $U$ spanned by $u_1,\dots,u_n$, and
the subspace $V_s$ of $V$ spanned by $v_1,\dots,v_n$ can be shown to be
unique.  In (\ref{eq:tensor}), any basis of $U_s$ can be used as $(u_i)$, and once $(u_i)$ are chosen, there is only one tuple of $(v_i)$ which satisfies the equation.  The common dimension $n$ of $U_s$ and $V_s$ is the rank of the element $s$
in the tensor product space $U\otimes V$.

Now suppose $t*s = s*t$, with $s$ and $t$ nonzero.  We claim that
$U_s = U_t$ and $V_s = V_t$.  For if these are not true, assume
without loss of generality that $U_s \not\subseteq U_t$.  Let us
construct a basis of $U_s + U_t$ by starting with a basis of $U_s
\cap U_t$, taking bases of $U_s$ and $U_t$ which contain it, and
forming their union.  In particular, by the above "$\not\subseteq$"
condition, the basis of $U_s$ will contain some element $x$ not in
the basis of $U_t$.

 We will now express the common value of $t*s =
s*t$ in terms of the above basis of $U_s + U_t$ and arbitrary
bases for $Z$ and $V$, and will find that the expression for $s*t$ contains
terms $x\otimes \dots$ for the $x$ referred to above, but the
expression for $t*s$ contains no such terms.  That will be a contradiction. 

Let $\{u_i\}$ be the basis for $U_s+U_t$ just constructed.  Thus $x\in\{u_i\}$.   Let $\{v_j\}$ be a basis for $V$.
We can now express $s,t$ in terms of these bases grouping terms as follows:
\[s=\sum_i u_i\tsr\left(\sum_j\alpha_{ij} v_j\right),\quad t=\sum_l \left(\sum_k \beta_{kl} u_k\right) \tsr v_l\]
Now let us calculate
\[s*t=\sum_{il} u_i\tsr\mu\!\!\left(\sum_j\alpha_{ij}v_j,\sum_k\beta_{kl}u_k\right)\tsr v_l.\]
Since $t\neq 0$, for some value of $l$ the second argument of $\mu$ is non-zero in the above.  When $u_i=x$, the first argument of $\mu$ is also non-zero.
Since $\mu$ has no zero-divisors the value of $\mu$ for those arguments is non-zero.  It follows that regardless of the choice of basis for $Z$, $s*t$ will contain a tensor element of the type $x\tsr z\tsr v_l$ with a nonzero coefficient.  On the other hand, repeating this analysis with the roles of $s$ and $t$ switched, we see that the same is not true of $t*s$.

So $U_s = U_t$ and $V_s = V_t$.  Let us write $t = \sum_1^n
u_i\otimes v_i$ where $\{u_i\}$ and $\{v_i\}$ are appropriate bases
of these subspaces.  We can express $s$ using these same bases as $s
= \sum_{i=1}^n u_i\otimes (\sum_j \alpha_{ji}) v_i$, where
$((\alpha_{ij}))$ is a matrix over $k$.  (The matrix will actually
be invertible, because the elements $\sum_j \alpha_{ji} v_i$ in the
expression for $s$ must form a basis of $V_s$; but we won't need this
fact.)  Now since $k$ is algebraically closed, this matrix has an
eigenvalue $\lambda$.  Defining $s' = s - \lambda t$, we get
\(s'=\sum\limits_{i=1}^n u_i\tsr\sum_j(\alpha_{ij}-\lambda\delta_{ij})v_j\).
Now $V_{s'}$ is spanned by the vectors $ \sum_j(\alpha_{ij}-\lambda\delta_{ij})v_j$ ($i=1,\dots n$).   Since the matrix  $((\alpha_{ij}-\lambda\delta_{ij}))$ is singular, $V_{s'}\subsetneq V_s=V_t$.  
Thus $V_{s'}\neq V_t$, and of course  $s'$ again satisfies $s' * t = t * s'$.   If $s'\neq 0$, this would
contradict the result of the two preceding paragraphs.  So $s' = 0$;
i.e., $s = \lambda t$.
\end{proof}

%
%
\section{Commuting homogeneous elements}
The terminology and notation used in this section are introduced in Sections~\ref{fga} and~\ref{degFun}.  
\begin{defn}We say a product $w_1\cdot\ldots\cdot w_n$, where $w_i\in F$, is a {\em reduced factorization} or a {\em reduced product}  if no $w_i$ equals $1$, and the {\em word} $\bar{w_1}\cdot\ldots\cdot\bar{w_n}$ is reduced, in other words, if no cancellation takes place at the juncture of $\bar{w_i}$ and $\bar{w_{i+1}}$ for $i=1,\dots,n-1$.  In this section, we reserve the notation ``$\rprod$" for {\em reduced} products in $F$.

For $z,w\in F$, we say that $z$ is a prefix (suffix) of $w$ whenever the {\em word} $\bar z$ is a prefix (resp. suffix) of the {\em word} $\bar w$, or equivalently, whenever $z=1$, $z= w$, or for some $v\in F$, there is a reduced factorization $w=z\cdot v$ (resp.  $w=v\cdot z$).  By $z\leq w$ ($z< w$) we mean that $z$ is a prefix (resp. proper prefix) of $w$.     
\end{defn}
\begin{obs}
If $n<m$, and $v_1\rprod\hdots\rprod v_n$ and $v_n\rprod\hdots\rprod v_m$ are reduced factorizations, so is
\[v_1\rprod\hdots\rprod v_m.\]

\end{obs}
\begin{obs}\label{obs:red} If $v_1\rprod\hdots\rprod v_n$ and $w_1\rprod\hdots\rprod w_m$ are reduced factorizations and neither $\bar{v_n}$ nor $\bar{ w_1}$ cancel completely in $\bar{v_n}\bar{ w_1}$, then 
\[ v_1\rprod\hdots\rprod v_{n-1}\rprod(v_nw_1)\rprod w_2\rprod\hdots\rprod w_m\]
is a reduced factorization when $v_n w_1$ is regarded as one factor.  \end{obs}
Indeed, the first letter of $\bar{v_nw_1}$ is the same as the first letter of $\bar v_n$ while the last letter of $\bar{v_nw_1}$ is the same as the last letter of $\bar w_1$.

We intend to capitalize on this observation by introducing certain sets of homogeneous elements of $\AA$ in the following definition.  
Although these sets, denoted by $T(r)$, may seem quite restricted, we shall find that for every $\ww\in\AA\wozero$ homogeneous of positive degree, all sufficiently large powers $\ww^n$ belong to $T(r)$ for appropriate values of $r$; and this will allow us to apply the properties of $T(r)$ in studying commuting elements.
\begin{defn} \label{T} Let $r>0$.  Let
\begin{itemize}\item  $T_1(r)$ be the set of homogeneous elements $\ww\in\AA$ for which
\[h(\ww)\geq 2\left(r+ \max\{\;\; \abs{h(a)}\; \big|\; a\in\Alphabet \}\right); \]
\item
$T_2(r)$ be the set of homogeneous elements $\ww\in\AA$ for which
\[ \qquad h(\ww)>0 \mbox{ and } \left(\forall w\in  \supp(\ww) \;\;\;\forall p\leq w\right)\;\;\;\;\;\; -r < h(p) < h(\ww)+r;\]
\item $T(r)=T_1(r)\cap T_2(r)$.
\end{itemize}
We will omit $r$ from this notation when there is no risk of confusion.  
\end{defn}
\begin{figure}[htp]
    \centering
    \begin{subfigure}{0.6032\textwidth}
      \includegraphics[width=\textwidth]
      {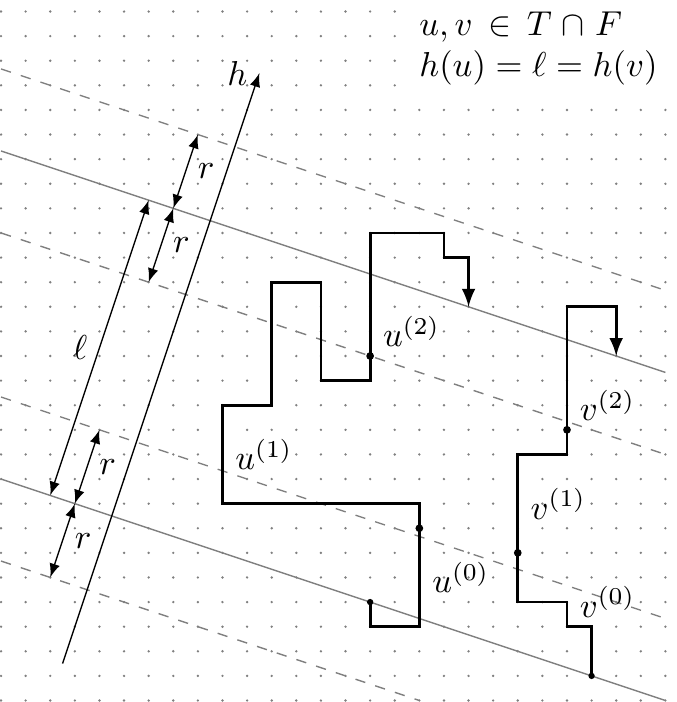}
      \caption{}
    \end{subfigure}\quad
    \begin{subfigure}{0.364\textwidth}
      \includegraphics[width=\textwidth]
      {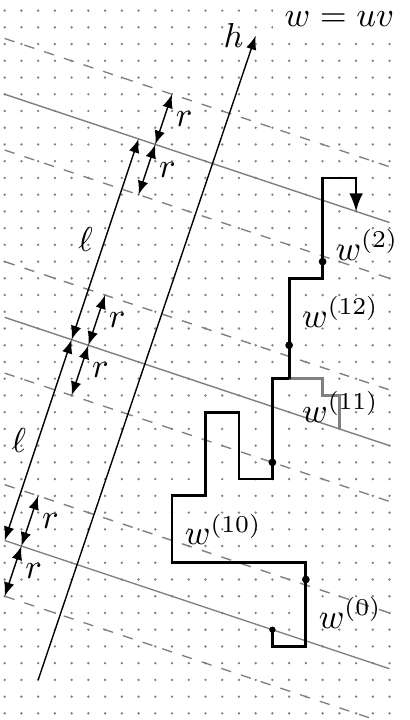}
      \caption{}
    \end{subfigure}
  \caption {Factorization of elements of $T\cap F$ when $\mbox{rank}\,F=2$. The group elements are depicted as walks in $F\ab$.} \label{fig:factorization}
\end{figure}


Figure~\ref{fig:factorization}\textsc{a} illustrates elements of $T\cap F$.   In that diagram, $F$ is free on generators $a$ and $b$.  In each path shown, every rightward, leftward, upward or downward $1$-unit step represents a factor $a$, $a^{-1}$, $b$, $b^{-1}$ respectively.  Moving along a path towards the arrowhead corresponds to reading the factors from left to right.  The function $h$ is represented by perpendicular projection onto the slanting line at the left.  The factorizations $u= u\zero\cdot u\one\cdot u\two$ and $v= v\zero\cdot v\one\cdot v\two$ to be discussed later in this section are shown by dots in the paths separating the indicated factors.  For visual clarity, the two expressions $u$ and $v$ are shown as paths with different initial points, so that their diagrams do not overlap, but both initial points are taken on the line $h=0$.

Let us fix an $r>0$.  We claim that the sets $T_1$, $T_2$, and $T$ are closed under multiplication.  It is obvious that  $T_1$ is closed under multiplication.  Let us show now that the same holds for $T_2$ .  Suppose   $\uu,\vv\in T_2$.   We want to show that $\uu\vv\in T_2$.  Let $u\in\supp(\uu)$ and $v\in\supp(\vv)$ and $p\leq uv$.   There are two possibilities:  $p\leq u$ or $p=uq$ where $q\leq v$.  In the first case we have
\[-r<h(p)<h(\uu)+r<h(\uu\vv)+r\]
while in the second
\[-r<h(q)<h(uq)=h(\uu)+h(q)< h(\uu)+ h(\vv)+r=h(\uu\vv)+r.\]
Thus $\uu\vv\in T_2$.  Lastly, it follows that $T$ is closed under multiplication as well.

Our next goal is to obtain certain reduced factorizations of  elements of $T\cap F$.  

Define
\begin{align*}
\Ocal&= \{ \;w \in F \;| \; h(w) \geq r,\;\;\;\;\ \left(\forall p <w\right)\;\;\;\; \abs{h(p)} <r  \;\}\\
\widetilde\Ocal&= \{ \;w \in F\; | \;  h(w) \geq r,  \;\;\left(\forall\mbox{ proper suffix } s \mbox{ of } w\right)\;\;\;\;\abs{h(s)} <r\;\}.
\end{align*}
It helps to notice that $\widetilde\Ocal$ is the reversal of $\Ocal$ and so statements about $\Ocal$ translate into dual statements about $\widetilde\Ocal$.  By comparing $h$-degrees we get these 
\begin{lem}[Properties of $\Ocal$ and $\widetilde\Ocal$]\label{propO}$ $
\begin{itemize}
\item No element in $\Ocal$ is a prefix of another element in $\Ocal$ 
\item No element in $\widetilde\Ocal$ is a suffix of another element in $\widetilde\Ocal$
\item No element in $\Ocal^{-1}:=\{u^{-1}\;|\; u\in\Ocal\}$ is a suffix of an element in $\widetilde\Ocal$
\item No element in $\widetilde\Ocal^{-1}:=\{v^{-1}\;|\; v\in\widetilde\Ocal\}$ is a prefix of an element in $\Ocal$
\end{itemize}
The last two properties can be equivalently stated as:
\begin{itemize}
\item  For any $u \in\Ocal$  and $v\in \widetilde\Ocal$, neither $\bar{u}$ nor $\bar{v}$ cancels completely in the product $\bar{vu}$
\end{itemize}
\end{lem}
\begin{proof}$ $
\begin{itemize}
\item If $v, w\in\Ocal$ and $p< v$,  then $h(w)\geq r$ but $h(p)<r$, hence $w\neq p$. 
\item Dual to the previous one.
\item If $t\in \Ocal^{-1}$, then $h(t)\leq -r$.  But if $s$ is a suffix, proper or not, of an element of $\widetilde{\Ocal}$, then $h(s)>-r$.  Hence $t\neq s$.
\item  Dual to the previous one.
\end{itemize}
\end{proof}
Let $w\in T\cap F$.  Define $w\zero$ to be the prefix of $w$ of the shortest word length such that $h(w\zero)\geq r$.  Likewise define $w\two$ to be the suffix of $w$ of the shortest word length with $h(w\two)\geq r$.  By the definition of $T_1$, $w\zero$ and $w\two$ exist.  We claim that there is a gap between them, i.e. that there is a $w\one\in F\without\{1\}$,  for which $w=w\zero\rprod w\one\rprod w\two$ is a reduced factorization (Figure~\ref{fig:factorization}\textsc{a}).   

For otherwise $s=\left(w\zero\right)^{-1}\!\!\!w$ would be a suffix of $w\two$ and then by the definition of $w\two$,
\[h(s)< r + \max\{\; \abs{h(a)} \big| a\in A\}.\]  
On the other hand, by the definition of $w\zero$, 
\[h(w\zero)< r+ \max\{\; \abs{h(a)} \big| a\in A\}.\]
  Putting these together we get 
  \[h(w)= h(w\zero s)= h(w\zero)+h(s)< 2\left(r+ \max\{\; \abs{h(a)} \big| a\in A\}\right),\]
   a contradiction with the definition of $T_1$.  
   
   Since $w\in T_2$, it follows from the definition of $T_2$ that $w\zero\in \Ocal$ and $w\two\in \widetilde{\Ocal}$.  Finally it is evident from the first two Properties of $\Ocal$ and $\widetilde{\Ocal}$ in Lemma~\ref{propO}, that the reduced factorization of $w$ we have constructed,  
\begin{equation}\label{eq:fact}
w= w\zero\rprod w\one\rprod w\two,
\end{equation}
is the only one with the property that 
\[w\zero\in \Ocal,\; w\two\in \widetilde\Ocal.\]

Now take two elements $u,v\in F\cap T$.  We have the reduced factorizations $u\zero\rprod u\one\rprod u\two$ and $v\zero\rprod v\one\rprod v\two$.   Next, by the last Property of $\Ocal$ and $\widetilde{\Ocal}$, neither $\bar{u\two}$ nor $\bar{v\zero}$ cancel completely in $\bar{u\two v\zero}$.
Therefore, by Observation~\ref{obs:red}, we have
the following reduced factorization of $uv$ (into $5$ factors with $u\two v\zero$ as one):

\begin{equation}\label{redfactw}
uv=u\zero \rprod u\one\rprod \left(u\two v\zero\right)\rprod v\one \rprod v\two.
\end{equation}
Since $u\zero\in\Ocal$ and $v\two\in\widetilde{\Ocal}$, it follows from uniqueness of~(\ref{eq:fact}) with $w=uv$, that 
\begin{equation}\label{prefsuf}
(uv)\zero = u\zero, \quad(uv)\one=u\one\rprod \left(u\two v\zero\right)\rprod v\one, \quad \mbox{and}\quad(uv)\two= v\two.
\end{equation}
Moreover, we have
\begin{lem-def} \label{Tlemma}Let $u,v\in F\cap T$, $w=uv$, and $\ell=h(u)$.
\begin{enumerate}[(a)]
\item Define the following elements:
\[w\onezero=u\one,\quad w\oneone=u\two v\zero,\quad w\onetwo=v\one.\]
\item $w\onezero$, $w\oneone$, and $w\onetwo$ depend only on $w$ and $\ell$.  In other words, if $f, g \in F\cap T$ satisfy $fg= w$ and $h(f)=\ell$, then 
\[f\one=u\one,\quad f\two g\zero=u\two v\zero,\quad g\one=v\one.\]
\item We have a reduced factorization:
\[w\one= w\onezero\rprod w\oneone\rprod w\onetwo.\] 
\end{enumerate}
\end{lem-def}
\begin{proof} (b):  Let $p= u\zero u\one = w\zero w\onezero$ and $q=u v\zero= w\zero w\onezero w\oneone$.  The factorization (\ref{redfactw}) shows that $p<q<w$.  Now we observe that ``$w$ crosses the interval $(\ell-r,\ell+r)$ only once; the crossing begins with $p$ and ends with $q$" (Figure~\ref{fig:factorization}\textsc{b}).   To be precise,
\begin{align*}
&\mbox{this holds...} & \mbox{because...}&\\
t< p &\implies h(t)<\ell+r,       &\mbox{( }t< u \mbox{ and } u\in T_2)&\\
t=p & \implies h(t)\leq\ell-r,    & \mbox{(} u=p\rprod u\two \mbox{ and } u\two\in\widetilde\Ocal)&\\
p<t<q &\implies \ell-r< h(t) <\ell+r, & \mbox{ (either } t^{-1}u \mbox{ is a proper suffix of } u\two\in \widetilde{\Ocal}&\\
&  &\mbox{ or } u^{-1}t < v\zero\in\Ocal)&\\
 t=q & \implies h(t)\geq \ell+r, & \mbox{( } q= uv\zero \mbox{ and } v\zero\in \Ocal)&\\
 t> q &\implies h(t)> \ell-r,  &  \mbox{( }u^{-1}t \leq v \mbox{ and } v\in T_2.)&\\
 \end{align*}
 
One can see from these conditions that $p$ is the prefix of $w$ of the longest word length which satisfies $h(p)\leq \ell-r$, while $q$ is the prefix of $w$ of the shortest word length satisfying $h(q)\geq \ell+r$.   Evidently $p$ and $q$ depend only on $w$ and $\ell$.  But with the aid of (\ref{prefsuf}),
we can write
\[w\onezero=\left(w\zero\right)^{-1}p,\quad w\oneone= p^{-1}q,\quad w\onetwo=q^{-1}w\left(w\two\right)^{-1}.\]
Hence $w\onezero$, $w\oneone$, and $w\onetwo$ depend only on $w$ and $\ell$.

(c)  This is the middle equation in~(\ref{prefsuf}).
\end{proof}

While so far we have looked at elements of $T\cap F$, now let us see what this lemma implies about general elements of $T$.   Suppose $ \uu,\vv \in T$ and $\ww=\uu\vv$.  Denote $\ell=h(\uu)$.  We have well-defined decompositions into sums obtained by grouping terms (where $a$ and $b$ run through $F$ but only finitely many terms are non-zero):
\begin{align*}
\uu&=\sum_a\uu_a, &\mbox{  where  }  &\forall u \in \supp(\uu_a),\quad u\one=a\\
\vv&=\sum_b\vv_b  &\mbox{  where  }  &\forall v \in \supp(\vv_b),  \quad v\one=b\\
\ww&=\sum_{a,b} \ww_{a,b} &\mbox{ where }&\forall w\in\supp(\ww_{a,b}),\quad  w\onezero=a \mbox{ and } w\onetwo=b.
\end{align*}
 The last decomposition depends only on $\ww$ and $\ell$ and not on the factors $\uu$ and $\vv$.
From $\ww=\uu\vv$ and Lemma-Definition~\ref{Tlemma} parts (a) and (b), we get 
\[\ww_{a,b}=\uu_a\vv_b.\]  

In particular, and this is the case of interest for us at the moment, suppose $h(\uu)=h(\vv)$ and $\uu\vv=\vv\uu$.  Then we arrive at
\begin{equation}\label{eq:uavb}
(\forall a,b)\quad\uu_a\vv_b= \vv_a\uu_b.
\end{equation}
This right away gives two things.  First, by letting $b=a$ in the above,
\begin{equation*} (\forall a)\quad\uu_a\vv_a=\vv_a\uu_a.\end{equation*}
  Second, $\uu\in T$ implies $\uu\neq 0$ and so there is a $b$ such that $\uu_b\neq 0$.  With that value of $b$ and with the aid of the fact that $\AA$ is an integral domain, (\ref{eq:uavb}) shows that $\uu_a=0\implies\vv_a=0$.  Similarly $\vv_a=0\implies\uu_a=0$.  Thus
  \begin{equation*}
  (\forall a)\quad \uu_a=0 \iff \vv_a=0.
 \end{equation*} 
 As $ \{u\one| u\in \supp(\uu)\}= \{a | \uu_a\neq 0\}$ and  $ \{v\one| v\in \supp(\vv)\}= \{a | \vv_a\neq 0\}$, we get
\begin{equation*} \{u\one| u\in \supp(\uu)\} =\{v\one| v\in \supp(\vv)\}.
\end{equation*}

\begin{prop} \label{hg} Suppose $\uu,\vv \in \AA$ are homogeneous of the same non-zero degree with respect to a homomorphism ${h: F\to (\RR,+)}$.   
If $\uu\vv=\vv\uu$, then $\uu=\lambda \vv$ for a scalar $\lambda$.
\end{prop}
\begin{proof}  
 Let us suppose without loss of generality that $h(\uu)=h(\vv)>0$, for otherwise we can replace the homomorphism $h:F\to(\RR,+)$ with $-h$.  

Suppose first that $\uu,\vv\in T(r)$ for some $r>0$.

Denote $S=   \supp(\uu) \bigcup  \supp(\vv)$.  For $a\in S\one:=\{ w\one | w\in S\}$, let $S_a= \{ w \in  S \; | w\one = a\}$, obtaining a partition of $S$: $S=\bigsqcup_a S_a$.  

Recapping the discussion prior to the present proposition, let us decompose uniquely
\[\uu= \sum_a \uu_a \mbox{ where }  \supp(\uu_a)\subseteq S_a, \;\;\;\;\;\;\;\;\;\;\;\;\;\;\vv= \sum_b \vv_b \mbox{ where }  \supp(\vv_b)\subseteq S_b,\]
arriving at $\forall a,b$,
 \begin{equation*} \uu_a\vv_b=  \vv_a \uu_b,\quad\uu_a\vv_a=\vv_a\uu_a,\quad \uu_a\neq0 \iff \vv_a\neq 0 \iff a\in S\one.\end{equation*}

Let us now set the stage for an application of Lemma~\ref{linAlgLem}. 

Fix any $a\in S\one$.  Let \[S_a\zero:=\{w\zero | w\in S_a\}, \qquad S_a\two:=\{w\two | w\in S_a\},\] and 
\begin{align}
U &= \vecspan{k}{\!\left(S_a\zero\right)}\\
V &= \vecspan{k}{\!\left( S_a\two\right)}  \qquad\mbox{(Note that $U, V\subset\AA$)}\\
Z&=VU\subset \AA
\end{align}
 and define
\[\mu:V\tsr U \to Z\]
to be the multiplication from $\AA$.  

Let us verify that $\mu$ has no zero divisors over $\bar{k}$ (cf. Lemma~\ref{linAlgLem}).   Tensoring by $\bar{k}$ the inclusions
\[U, V, Z\subset \AA,\]
we get 
\[\bar{k}\tsr U,\; \bar{k}\tsr V, \;\;\bar{k}\tsr Z \;\subset \;\bar{k}\tsr(\gring{k}{F}) \cong \gring{\bar{k}}{F}.\]
 Moreover, $\bar{k}\tsr \mu$ coincides with the multiplication in $\bar{k}F$ under these inclusions.   A free group algebra over any field is an integral domain.  In particular, so is $\bar{k}F$.  Hence $\mu$ has no zero divisors over $\bar{k}$.

Thus Lemma~\ref{linAlgLem} applies, but in order to benefit from it, we need to establish a couple of isomorphisms.   

Observe that for any $u\in S_a\zero \subseteq\Ocal$ and
$v\in S_a\two\subseteq\widetilde\Ocal$,  
\[u\rprod a\rprod v\]
 is a reduced factorization because $u\rprod a$ and $a\rprod v$ have this property.  
 

This observation shows that $u$ is a prefix of $uav$.  As $u\in\Ocal$ and as no element of $\Ocal$ is a prefix of another (the first of the Properties of $\Ocal$ and $\widetilde\Ocal$), $u$ is {\em the unique} prefix of $uav$ contained in $\Ocal$.  Likewise $v$ is {\em the unique} suffix of $uav$ contained in $\widetilde{\Ocal}$.  We conclude that the map
\[S_a\zero \times S_a\two\to S_a\zero a S_a\two,  \;\;\;\;   (u, v) \to uav\] 
is injective. As it is also clearly surjective, it is bijective. Passing from a bijection of bases to an isomorphism of vector spaces, we get 
\[U\tsr V\isoto UaV.\] 
 The second isomorphism which we need to demonstrate is
\begin{equation}\label{eq:UZV}
U\tsr Z\tsr V \isoto UaVUaV\subset \AA
\end{equation}
\[ u\tsr z\tsr v \mapsto uazav.\]
Looking at the bases, it suffices to show that the map
\[S_a\zero\times S_a\two S_a\zero \times S_a\two \to F\]
\[ (u, gf, v)\mapsto uagfav\]
is injective.  To show that, we start with $w=uagfav$ and from it we will now recover $u, v$ and $gf$.   We have the reduced factorizations $u\rprod a\rprod g$ and $f\rprod a\rprod v$.  By the last Property of $\Ocal$ and $\widetilde{\Ocal}$, neither $\bar{g}$ nor $\bar{f}$ cancel completely in $ \bar{gf}$.   Therefore by Observation~\ref{obs:red},
\[w= u\rprod a \rprod(gf)\rprod a\rprod v\]
is a reduced factorization (with $gf$ as a single factor).  By the first two Properties of $\Ocal$ and $\widetilde\Ocal$, $u$ is the only prefix of $w$ contained in $\Ocal$ and $v$ is the only suffix of $w$ contained in $\widetilde{\Ocal}$.   So we have recovered $u$ and $v$.  Then we recover $gf$ via  $gf= (ua)^{-1}w(av)^{-1}$. 
This shows that the above map is indeed injective and therefore (\ref{eq:UZV}) holds.

Finally observe that $\uu_a,\vv_a\in UaV\cong U\tsr V$.  

For every $a\in S\one$, we can now apply Lemma~\ref{linAlgLem} taking the images of $\uu_a$ and $\vv_a$ in $U\tsr V$ as $t$ and $s$.   Recall that for $a\in S\one$,  both $\uu_a$ and $\vv_a$ are non-zero.  So the lemma gives
  \[\vv_a= \lambda_a \uu_a.\]  
It remains to show that $\lambda_a$ is independent of $a$.  Indeed,
\[\uu_a \vv_b = \vv_a\uu_b,\]
\[\lambda_b\uu_a \uu_b = \lambda_a\uu_a\uu_b\]
and since $\AA$ is an integral domain, we can cancel $\uu_a\uu_b$ and conclude that $\lambda_a=\lambda_b$ for all $a,b$.

We have just shown that the conclusion of the present proposition holds when $\uu,\vv\in T$.  Now let us consider the general case.  Thus, now $\uu$ and $\vv$ are arbitrary commuting $h$-homogeneous elements of $\AA$ with $h(\uu)= h(\vv)$. 
Moreover as mentioned in the beginning of the proof, we can assume that $h(\uu)=h(\vv)>0$.

We will see now that for some $r>0$ and some $n>0$,  $\uu^n,\vv^n\in T(r)$.

By the definition of $T_2$ and by finiteness of supports, for $r>>0$ we have  $\uu,\vv\in T_2(r)$.   Let us fix such $r$.  It follows from the fact that $T_2$ is closed under multiplication, that $\forall n>0$, $\uu^n,\vv^n\in T_2$.   On the other hand, from the definition of $T_1$ and the Archimedean property of $\RR$ it follows that for $n>>0$, we also have that
$\uu^n, \vv^n\in T_1$.   For any such $n$, it follows that $\uu^n, \vv^n\in T$.

 Since $\uu$ and $\vv$ commute, $\uu^n$, $\vv^n$ commute as well.   By what we have just shown,  
 \begin{equation}\label{eq:nthroot}
 \uu^n= c \vv^n, \quad\quad c\in k.
 \end{equation}
 Regarding $\gring{k}{F}$ as a subring of $\gring{\bar{k}}{F}$, we can work in $\gring{\bar{k}}{F}$.  Since $\gring{\bar{k}}{F}$  is an integral domain and $\uu$ and $\vv$ commute, the sub-algebra $\bar{k}[\uu,\vv]\subseteq\gring{\bar{k}}{F}$ is a commutative integral domain.  This commutative integral domain is contained in its field of fractions $\Frac\bar{k}[\uu,\vv]$.   Equation  (\ref{eq:nthroot}) implies that $\uu/\vv$ is an $n$-th root of $c$ in $\Frac\bar{k}[\uu,\vv]$, and since $\bar{k}$ is algebraically closed,  it is algebraically closed in that overfield; so the $n$-th root of $c$ in that overfield must lie in $\bar{k}$; so $\uu\in\bar{k}\vv$.  Hence, since within $\gring{\bar{k}}{F}$, $\uu$ and $\vv$ have coefficients in $k$,  we must have $\uu\in k\vv$.  
\end{proof}

\begin{cor} \label{deg0} 
Suppose $\uu,\vv\in \AA\without k$ and $\uu\vv=\vv\uu$.  Let $h:F\to (\RR,+)$ be a group homomorphism and let $F\ab$ denote the abelianization of $F$ (with the group operation written additively).
\begin{compactenum}[(a)] 
\item If $\uu$ is $h$-homogeneous of degree $0$,  then $\vv$ is also $h$-homogeneous of degree $0$.  
\item If $\uu$,$\vv$  are arbitrary, not necessarily homogeneous,  the images of $\supp(\uu)$ and $\supp(\vv)$ in $\QQ\tsr_\ZZ F\ab$  span the same $\QQ$-vector space. 
\item If $\uu$, $\vv$ are $h$-homogeneous, then $\QQ h(\uu)=\QQ h(\vv)$.
\item If $\uu, \vv$ are $h$-homogeneous, $h(\uu)>0$,  $h(\vv)<0$, then $\uu$ and $\vv$ are monomials.
\end{compactenum}
\end{cor}

\begin{proof}

(a)  Suppose the contrary.  Then either the lowest- or highest-degree homogeneous component of $\vv$ has a non-zero degree.  Denote that homogeneous component by $\ww$.  Then we have:
\begin{align*}
& \;\;\;\;\;\; \;\;\;\;\uu\ww=\ww\uu\\
&  \implies(\uu\ww)\ww=\ww(\uu\ww)\\
& \implies \uu\ww=\lambda\ww \mbox{ for some }\lambda\in k  \mbox{  (by Proposition~\ref{hg}) }\\
&  \implies \uu = \lambda \mbox{ (because $\AA$ is an integral domain)}
\end{align*}
which is a contradiction.

(b)  Apply (a) to all possible degree functions $h$ which vanish on one of the two supports.

(c)     Let 
\[q: F\to \QQ \tsr_\ZZ F\ab\] 
be the canonical homomorphism.  Then $h$ factors through $q$: $h= \eta\circ q$.

Let $V= \QQ\, q(\supp(\uu))=\QQ\, q(\supp(\vv))$ be the common (due to part b) $\QQ$-span of the images of $\supp(\uu)$ and $\supp(\vv)$ in $\QQ\tsr_\ZZ F\ab$.  Then 
\[\eta(V)= \eta(\QQ \,q(\supp(\uu)))= \QQ \eta(q(\supp(\uu)))= \QQ h(\supp(\uu))= \QQ h(\uu),\]
where the last equality holds because $\uu$ is $h$-homogeneous.   Likewise, \( \eta(V)= \QQ h(\vv)\).

(d)  By (c), we can write $h(\uu)/h(\vv)= -n/m$ with $n,m$ positive integers.  Then $h( \uu^m\vv^n)=0$, and, since $\uu^m\vv^n$ commutes with $\uu$, by (a), $\uu^m\vv^n\in k$.   Moreover, since $\AA$ is a domain, $\uu^m\vv^n\neq 0$.  Hence $\uu,\vv$ are invertible, and all invertible elements in $\AA$ are monomials (cf. Section~\ref{fga}).
 \end{proof}

%
%
\section{Centralizers}
\begin{lem}\label{lem:val} Let $C$ be the centralizer of an element in $\AA\without k$.  If $h$ does not vanish on $C\wozero$ then $h$ restricted to $C$ is a degree function over $k$ which is discrete (cf. Definition~\ref{degFun}) and which satisfies:
\begin{equation}\label{eq:totallyRamified2}
     \left(\forall  \uu,\vv \in C\wozero\right)\quad    h(\uu)= h(\vv)                \implies \left(\exists \lambda \in k\right)\;\; h(\uu-\lambda \vv)<h(\uu).\quad \mbox{(cf.~(\ref{eq:totallyRamified}))}
\end{equation}
\end{lem}
\begin{proof}   

As shown in Section~\ref{degFun}, $h$ is a degree function on  $\AA$ and so $h\big|C$ is a degree function on $C$.  We just need to show that $h{\big|C}$ is discrete and has property~(\ref{eq:totallyRamified2}).

Before we continue, let us recall useful notation (cf. Section~\ref{degFun}): for any element $\uu\in \AA\wozero$, we denote by $\uu_h$ the highest-degree homogeneous component of $\uu$.  Recall also that 
\begin{equation}\label{eq:degLeadTerm}
h(\uu)=h(\uu_h).
\end{equation}
Let $\ww\in C\wozero$ be such that $h(\ww)\neq 0$.   Note that this implies that $\ww_h\not\in k$.

First, let us show that $h$ is discrete on $C$.  Let $\vv\in C\wozero$ be another element with $h(\vv)\neq 0$.   This implies that $\vv_h\not\in k$.  Since $\ww$ and $\vv$ commute, $\ww_h, \vv_h$ commute as well.  Now, applying Corollary~\ref{deg0}c to $\ww_h$ and $\vv_h$, we see that $h(\vv_h)\in \QQ h(\ww_h)$.  Therefore by (\ref{eq:degLeadTerm}),  $h(\vv)\in\QQ h(\ww)$. Thus, $h(C\wozero)\subseteq \QQ h(\ww)$.  So 
\begin{equation}\label{eq:discreteness}
h(C\wozero)\subseteq \QQ h(\ww)\bigcap h(F).
\end{equation}
Since $F$ is finitely generated, $h(F)$ is finitely generated abelian.  The right-hand side of (\ref{eq:discreteness}), being a subgroup of $h(F)$, is also finitely generated, but, given that it is also a subgroup of $\QQ h(\ww)\cong \QQ$, the right-hand side is cyclic.   

Let us now establish property~(\ref{eq:totallyRamified2}).  Suppose first $h(\uu)=h(\vv)=0$.  Then it is enough to show that $\uu_h,\vv_h\in k$.   Since $\ww_h$ and $\uu_h$ commute,  by Corollary~\ref{deg0}a, it follows that $\uu_h\in k$.  For the same reason, $\vv_h\in  k$.  

Now suppose $h(\uu)=h(\vv)\neq 0$.  Then $\uu_h$ and $\vv_h$ commute and the property follows directly from Proposition~\ref{hg}. 
\end{proof}
To set the notation of the next proposition and its corollaries, let $\uu\in \AA\without k$ and let $C$ be the centralizer of $\uu$.  Let $H<F$ be the subgroup of $F$ generated by the union of the supports of all elements of $C$:
\[ H = \big\langle \bigcup_{\vv\in C}\supp(\vv)\big\rangle.\] 
Let us call $H$ the {\em centralizer-supporting group of $\uu$}.
\begin{prop}\label{prop:centralizerSupportGroup}  
The centralizer-supporting group $H$ of $u\in \AA\without k$  is finitely generated.  Moreover, the image of $\supp(\uu)$ in $H\ab$, the abelianization of $H$, generates a subgroup of $H\ab$ whose rank equals the rank of $H$.  
\end{prop}
\begin{proof} Being a subgroup of $F$, the group $H$ is
free; and the centralizer of $\mathbf{u}$ in $\gring{k}{H}$ is clearly again $C$.  Now $\mathrm{Supp}(\mathbf{u})$ is finite, hence contained
in a finitely generated free factor of $H$, hence by Lemma~\ref{finiteAlphabet}, that
free factor is all of $H$, so $H$ is finitely generated.  

If the subgroup of $H\ab$ generated by the image of $\mathrm{Supp}(\mathbf{u})$
in $H\ab$ had rank less than $\rank(H) = \rank(H\ab)$,
that image would lie in the kernel of a nonzero homomorphism
$H\ab\to\mathbb{Z}$, so by Corollary~\ref{deg0}a, the supports of all elements
of $C$ would lie in that kernel, contradicting the assumption that
they generate $H$.
\end{proof}

In the following corollaries and below, when we say that an element of $\AA$ is {\em supported on a cyclic group} we simply mean that its support is contained in a cyclic subgroup of $F$.
\begin{cor} \label{Laurent}Suppose that $\uu\in \AA\without k$
is supported on a cyclic group.
Then the centralizer of $\uu$ in $\AA$ is of the form $\gring{k}{Z}$ where $Z$ is the largest cyclic subgroup of $F$ supporting $\uu$.\qed
\end{cor}
\begin{cor}\label{cyclic}  Suppose that $\uu\in \AA\without k$ is {\em not} supported on a cyclic group.  
Then the image of $\supp(\uu)$ in $H\ab$, the abelianization of $H$, generates a subgroup of $H\ab$ whose rank is at least $2$.  \qed
\end{cor}

We now arrive at our main result.
\begin{thm}\label{theorem} \label{thm}Suppose that $\uu\in \AA$ 
is {\em not} supported on a cyclic group.
Let $C$ be the centralizer of $\uu$.   Then $C$ is the affine coordinate ring of the complement of a $k$-point in a proper nonsingular curve over $k$.

\end{thm}
\begin{proof}
As discussed in Section~\ref{centralizerBasics}, $C$ is a commutative integrally closed integral domain.
Our proof will consist of showing that $C$ admits a discrete degree function $d$ which satisfies condition (i) of Lemma~\ref{characterization}.   For in that case, that $C$ is integrally closed will imply that the curve $\thecurve$ in condition (iv) of the lemma is nonsingular.
  
 Since $\supp(\uu)$ is not contained in a cyclic group, by Corollary~\ref{cyclic}, we can, if necessary, replace $F$ by the centralizer-supporting group of $\uu$ (cf. above Proposition~\ref{prop:centralizerSupportGroup}) and thereafter assume that the image of $\supp(\uu)$ in $F\ab$, the abelianization of $F$, generates a subgroup of rank at least $2$.  
 
 Under this assumption, we will construct a group homomorphism $h: F\to (\RR,+)$, (cf. Section~\ref{degFun}) with respect to which $h(\uu)>0$ and the highest degree homogeneous component $\uu_h$ of $\uu$ is not a monomial (cf. Section~\ref{degFun}).  Although it is not necessary, we can make the resulting degree function $h$ discrete without extra effort. 
 
 We will use the following geometric fact: If $K$ is a bounded convex polytope (=the convex hull of a finite set in $\RR^s$), the intersection of all the facets (=faces of codimension $1$) of $K$ is empty.   
 
 To prove this fact, suppose $O\in \RR^s$ belongs to every facet of $K$.  Let us translate $K$ so that $O$ is at the origin.  We can assume without loss of generality that $K$ is not contained in a proper subspace of $\RR^s$ for otherwise we could pass to such a subspace.   
 
 Since $K$ is convex, it is an intersection of finitely many half-spaces.  In other words, there are finitely many $\RR$-linear forms $l_i: \RR^s\to \RR$ and constant terms $b_i\in\RR$, such that
\begin{equation*}
K=\left\{ x\in \RR^s \Big| (\forall i)\;\;l_i(x)\leq b_i\right\}.
\end{equation*}
 Assuming that none of the inequalities $l_i(x)\leq b_i$ above are redundant, facets  of $K$ are the sets
$Y_i = \left\{ x\in K \Big|  l_i(x) =  b_i\right\}.$
Since $O$ belongs to every facet of $K$, it follows that $b_i= 0$ for all $i$.  Hence $K$ is a cone with the apex at $0$ in the sense that 
 \[ \lambda>0,\; x\in K\implies \lambda x \in K.\]
 This contradicts boundedness of $K$.
 
Let $F\ab$ be the abelianization  of $F$, with the group operation written additively.  Let 
\[q: F\to \RR\tsr_\ZZ F\ab\]
be the canonical homomorphism.  

We are now ready to construct $h$ with the properties stated in the 3rd paragraph of this proof.  Let $V$ be the $\RR$-span of $q(\supp(\uu))$ in $\RR\tsr _\ZZ F\ab$ and let $K$ be the bounded polytope which is the convex hull of 
\[q(\supp(\uu))\cup\{0\}.\]
By our assumption, $\dim K=\dim V\geq2$.   By the above fact about bounded convex polytopes, $K$ has a facet $Y$ which does not contain the origin $0$.   Choose  an $\RR$-linear form $l:V\to \RR$ and $b\in \RR$ so that
\begin{equation}\label{eq:facetY}
l\leq b \mbox{ on } K \mbox{ and } l=b \mbox{ on } Y.
\end{equation}
It follows that $l<b$ on $K\without Y$, and since $0\in K\without Y$,  we get $0< b$.   

Now extend $l$ to an $\RR$-linear form $\eta:\RR\tsr_\ZZ F\ab\to \RR$ anyhow and define
\[h= \eta\circ q.\]
 
Let us verify that $h$ has the desired properties.  We continue working inside $V$.  
 Since $K$ is the convex hull of $q(\supp(\uu))\cup\{0\}$, and $0\not\in Y$, it follows that the vertices of $Y$ are in $q(\supp(\uu))$.  Therefore, (\ref{eq:facetY}) implies that
  \[h(\uu)=\max\{ h(u)\big| u\in \supp(\uu)\}=\max\{ l(x)\big|  x\in q(\supp(\uu)) \}= b>0.\]
Now, if we write
\[\uu=\sum_{u} a_u u,\]
  we get 
 \[\quad \uu_h = \sum_{q(u)\in Y} a_u u.\]
Hence $\supp(\uu_h)= \{u\in \supp(\uu) \big| q(u)\in Y\}$, which implies that  the vertices of $Y$ are not only in $q(\supp(\uu))$, but also in $q(\supp(\uu_h))$.
Since $\dim K\geq 2$, we have $\dim Y\geq 1$, and thus $Y$ has at least $2$ vertices. 
Therefore, $\supp(\uu_h)$ contains at least $2$ elements, or, in other words, $\uu_h$ is not a monomial.
 
 
We now define our degree function $d$ to be the restriction of $h:\gring{k}{F}\to \RR\cup\{\infty\}$ to $C$.   Let us verify that $d$ satisfies condition (i) of Lemma~\ref{characterization}. 
Lemma~\ref{lem:val} applies and only leaves us to show that $h\geq 0$ on $C\wozero$.  So let $\ww\in C\wozero$.  Since $\ww_h$, the leading homogeneous component of $\ww$, commutes with $\uu_h$ and $h(\uu_h)>0$, by part (d) of Corollary~\ref{deg0}, $h(\ww)\geq0$.
\end{proof}
The following says that in the special case of a homogeneous element, the nonsingular curve in the conclusion of the theorem is a line.
\begin{prop}\label{prop:hg} If $\uu$ is homogeneous of non-zero degree, and not a monomial, then its centralizer is isomorphic to the polynomial ring $k[t]$.
\end{prop}

\begin{proof}After replacing the group homomorphism $h:F\to(\RR,+)$ by $-h$ if necessary, we can assume without loss of generality that $h(\uu)>0$.   Let $C$ be the centralizer of $\uu$.  As $\uu$ is homogeneous of non-zero degree and yet not a monomial, $u$ cannot be supported on a cyclic group.  Thus from Theorem~\ref{thm} we know that $C$ satisfies the equivalent conditions of Lemma~\ref{characterization}.  Moreover,  Remark~\ref{rem:uniquePlace} (or the proof of the  theorem) shows that under our assumptions about $\uu$, the restriction of $h$ to $C$ can be used as the degree function $d$ that appears in the lemma.  Since $\uu$ is homogeneous, $C$ is in fact graded by $h$: indeed, if $\vv$ commutes with $\uu$, so does each homogeneous component of $\vv$. 

Let $\tt\in C$ be a homogeneous element of the smallest positive degree.  Let $\vv\in C$ be homogeneous.  If $h(\vv)=0$, by property~(\ref{eq:totallyRamified}) in Lemma~\ref{characterization}, $\vv\in k$.  Otherwise, since $h$ is discrete on $C$, there are coprime positive integers $n$ and $m$ for which $h(\vv^n)= h(\tt^m)$.  Then by property~(\ref{eq:totallyRamified}) again and homogeneity it follows that for some $\lambda\in k$, $\vv^n = \lambda \tt^m$.   Now $C$ is integrally closed (cf. Section~\ref{centralizerBasics}), therefore the equation $\vv^n = \lambda \tt^m$ implies that both $\vv$ and $\tt$ are scalar multiples of powers of
an element in $C$, and since $\tt$ has the smallest positive degree possible, that element can be taken to be $\tt$.   Thus $\vv= \lambda \tt^m$ and since $\vv$ was an arbitrary homogeneous element of $C$, we conclude that $C=k[\tt]$. Finally, by Lemma~\ref{characterization}, $\tt$ is transcendental over $k$.
\end{proof}
\section{Open Question}
Recall that in the Introduction~\ref{intro} we divided G. Bergman's proof of his theorem into two steps.  By analogy with the second step,  we would like to ask the following:

\begin{que} Does every finitely generated subalgebra $R\neq k$ of a free group algebra $\AA$ over $k$ admit a $k$-algebra homomorphism $f:R\to k[t,t^{-1}]$ into the ring of Laurent polynomials of a single variable  $k[t,t^{-1}]\cong k[\ZZ] $ which is nontrivial in the sense that  $f(R)\neq k$?
\end{que}

%
%
\section{Proof of Lemma~\ref{characterization}}\label{sec:proofCurves}
By a {\em function field of one variable} over $k$ we mean a finitely generated field extension $K/k$ of transcendence degree $1$.   By a {\em place} of $K/k$ we mean an equivalence class of discrete valuations (cf.~\cite[Ch 3, Definition 3.22]{Liu}) on $K$ which are trivial on $k$.

Recall that if $K$ and $L$  are  function fields of one variable over $k$, and $K\subset L$, then for any place $q$ of $K$,
  \begin{equation}\label{eq:funId}
  [L:K]= \sum_{p} e_{p/q}\, [k(p): k(q)]
  \end{equation}
where the summation is over all places $p$ above $q$, $e_{p/q}$ is the ramification index (cf.~\cite[Ch 7, Exercise 1.8]{Liu})  at $p$ and $k(p)$ and $k(q)$ are the residue fields at $p$ and $q$ resp. (cf~\cite[(4.8), p. 290)]{Liu}).   Clearly for any place $p$ above $q$,
\begin{equation}\label{eq:ramvsdeg}
e_{p/q}\leq [L:K].
\end{equation} 
When we say that $q$ is {\em totally ramified} in $L$  we mean that for some $p$ above $q$, 
\[e_{p/q}= [L:K].\]
Formula (\ref{eq:funId}) shows that this happens if and only if there is a sole place $p$ above $q$ and moreover $k(p)=k(q)$.  

\begin{proof}[Proof of Lemma~\ref{characterization}]$ $

$(i)\implies (ii)$: 
By property~(\ref{eq:totallyRamified}), all elements of $C$ of degree $0$ are in $k$.   Since we assume that $C\neq k$, it follows that $d$ is not identically zero on $C\wozero$.  Since $d$ is discrete, we can normalize $d$ so that it maps $C\wozero$ into $\ZZ$ but not into any proper subgroup of $\ZZ$.  Also, since $d(xy)= d(x)+d(y)$, $d(C\wozero)$ is closed under addition.  Given a set of positive integers $x_1,\dots,x_s$ with $g.c.d.(x_1,\dots, x_s)=1$, it is well-known that all sufficiently large integers can be written as linear combinations of the $x_i$ with non-negative integer coefficients.  This is best known in the context of the coin problem, which asks for the largest integer, known as the Frobenius number, which is {\em not} equal to such a linear combination (see for example, ~\cite{Erdos}).  Thus, by the finiteness of the Frobenius number, $d(C\wozero)$ contains all large enough integers.

Let $x\in C\without{k}$ and denote $n:= d(x)$.  Note that $n>0$.

$x$ is transcendental:  
 
Suppose $f(t)= a_0+a_1t +\dots + t^s$ is a non-zero polynomial over $k$.  Then $f(x)= a_0+ a_1x+\dots +  x^s \neq 0$ because, using Remark~\ref{rem:isosceles}, we obtain $d(f(x))=ns\neq d(0)=-\infty$.
  
$C$ is a finitely generated $k[x]$-module:  

We are reproducing here the argument in the proof of ~\cite[Proposition 2.2]{Bergman}. Let $e_1,\dots, e_n\in C$ be chosen so that each congruence class modulo $n$ contains exactly one $d(e_i)$ and this is equal to the minimal member of this class in  $d(C\wozero)$.  This is possible because we assume that $d\geq 0$ on $C\wozero$.  Let $M$ be the $k[x]$-module generated by $e_1,\dots, e_n$.  We claim that $M=C$.  Suppose the contrary and let $y\in C\without M$ be of the least $d$-degree.   Let $j$ be such that 
\(d(y)\equiv d(e_j)\; \;(\!\!\!\!\mod n).\)  
Then for some $s\geq 0$, $d(y)=d(x^s e_j)$ and by property~(\ref{eq:totallyRamified}), for some $\lambda\in k$, 
\[d(y-\lambda x^s e_j)<d(y).\]
Since $y$ has the least degree for an element outside of $M$,  $m= y-\lambda x^s e_j\in M$, and thus $y= m+\lambda x^s e_j\in M$, a contradiction.  Thus $M=C$ and therefore the $k[x]$-module $C$ is generated by $e_1,\dots, e_n$.

$\infty$ is totally ramified in $\Frac C$:  

Let $v_\infty$ denote the valuation of $k(x)$ at $\infty$.   Thus for any polynomial $f\in k[x]$,  
\[v_\infty(f) = -\deg f.\]
 Let $v$ denote the valuation on $\Frac C$ which coincides with $-d$ on $C$ (cf. Remark~\ref{rem:valuations}) and let $p$ denote the corresponding place of $\Frac C$.   Since $d$ maps $C\wozero$ into $\ZZ$ but not into any of its proper subgroups,  our $v$ sends $\Frac C\wozero$ surjectively onto $\ZZ$, i.e. $v$ is normalized.   Applying Remark~\ref{rem:isosceles}, $\forall f\in k[x]$,
 \[d(f) = n \deg f,\quad \mbox{ hence }   \quad v(f) =n\,v_\infty(f).\]
 Thus $v$ is an extension of $v_\infty$ to $\Frac C$ with ramification index $e_{p/\infty}=n$.  
  
 Since $C$ is generated as a $k[x]$-module by the $n$ elements $e_1,\dots, e_n$,  
 \[  e_{p/\infty}=n\geq [\Frac C: k(x)].\] 
After comparing this with (\ref{eq:ramvsdeg}), we conclude that  $\infty$ is totally ramified in $\Frac C$ and get as a side consequence that $e_1,\dots, e_n$ is a basis for $C$ as a free $k[x]$-module.
 
$(ii)\implies (iii)$:  Trivial.

$(iii)\implies (iv)$:   Let $U=\Spec C$.  Let $\thecurve$ be the proper completion of $U$ obtained by adding to $U$ the regular points corresponding to all the places of $\Frac C$ which are not over any point in $U$.  This is the {\em smallest proper completion of $U$} (cf.~\cite[Section 4.1 Exercise 1.17]{Liu} in the sense that any morphism from $U$ to a proper scheme  $Y$ 
extends uniquely to a morphism $\thecurve\to Y$.   The curve $\thecurve$ is reduced and irreducible.

The element $x$ determines a morphism of algebraic curves over $k$:  
\[\phi :U \to \mathbb{A}^{1}_k.\]

 As $\mathbb{A}^1_k\hookrightarrow\mathbb{P}^1_k$, and $\mathbb{P}^1_k$ is proper, $\phi$ extends uniquely to a morphism $\phi : \thecurve\to \mathbb{P}^1_k$: 
\begin{equation}
 \begin{tikzcd}
U \arrow[r, hook] \arrow[d, "\phi"]
& \thecurve \arrow[d, "\phi" ] \\
\mathbb{A}^1_k \arrow[r, hook]
& \mathbb{P}^1_k
\end{tikzcd}
\end{equation}
 Clearly $\phi^{-1}(\infty)\subseteq \thecurve\without U$, so $\phi^{-1}(\infty)$ consists of regular points.  Thus the local rings of points in the set $\phi^{-1}(\infty)$ are valuation rings of places of $\Frac C$ above $\infty$.    Since $\infty$ is totally ramified in $\Frac C$, by (\ref{eq:funId}), there is a unique place of $\Frac C$ above $\infty$, and moreover the residue field at that place is $k$.   Thus $\phi^{-1}(\infty)$  consists of a unique point $p\in \thecurve(k)$.
 
It remains to show that $U=\thecurve\without\{p\}$.  As we are about to see, this follows from the finiteness of $k[x]\inject C$.   We have
\[ U\inject \thecurve\without\{p\}\to \mathbb{A}^1_k\]
which gives
\[k[x]\inject \Ocal_\thecurve(\thecurve\without\{p\})\inject C.\]
Since $k[x]\inject C$ is finite, so is 
\begin{equation}\label{eq:finite}
\Ocal_\thecurve(\thecurve\without\{p\})\inject C.
\end{equation}
  Here we will use the following important fact:  Any curve over $k$ which is not proper is affine (\cite[Sec. 7.5, Ex. 5.5]{Liu}).   By this fact $\thecurve\without\{p\}$ is affine, and it follows from finiteness of~(\ref{eq:finite}) that $U\inject \thecurve\without\{p\}$ is finite and hence integral.   Hence, by the Going-up theorem for integral homomorphisms, it is  surjective.  Thus $U= \thecurve\without\{p\}$.  
  
$(iv)\implies (v)$:  Let $v$ denote the discrete valuation at $p$.   Suppose $v(x)\geq 0$ for some $x\in C$.   Then $x$ is a global section of $\thecurve$ and since $\thecurve$ is proper, it follows (see for example~\cite[Sec. 7.3.2, Corollary 3.18]{Liu}) that $v(x)=0$.  Thus $v$ is non-positive on $C$.    The residue field of $v$ is $k$ because $p$ is $k$-rational.

$(v)\implies (i)$:

Define $d$ to be the restriction of $-v$ to $C$.  From the definition of a discrete valuation it follows that $d$ is a discrete degree function (cf. Remark~\ref{rem:valuations}).
It remains to show property~(\ref{eq:totallyRamified}), which, when expressed in terms of $v$, becomes: 
  \begin{equation}\label{eq:tr2}
 (\forall x,y\in C\wozero)\quad    v(x)= v(y)                \implies \left(\exists \lambda \in k\right)\;\; v(x-\lambda y)>v(x).
\end{equation}
Let $R$ be the local ring of  $v$ and $t$ be a uniformizer, i.e, an element of $R$ with $v(t)=1$.  The quotient $R/tR$ is the residue field.  Thus $R/tR=k$.  We have for all $n$: 
\[v(x)\geq n \iff x\in Rt^n\]
\begin{equation}\label{eq:grdim1}
Rt^{n}\big/ Rt^{n+1}\cong R\big/ Rt = k.
\end{equation}
Hence property~(\ref{eq:tr2}) holds for all $x,y \in\Frac C$.  

The isomorphism (\ref{eq:grdim1}) also shows that every element $x\in\Frac C$ can be expanded into a Laurent series in $t$ with coefficients in $k$.  Then $v(x)$ is the degree of the lowest term.   Viewed from this vantage point, property~(\ref{eq:tr2}) is clear.
\end{proof}
%
%

\bibliography{mybibfile}

%

\end{document}